\documentclass[11pt]{amsart}

\usepackage{amsthm,amssymb,enumitem,graphicx,kbordermatrix,
framed,wasysym}

\newcommand{\ba}{\backslash}
\newcommand{\dash}{\nobreakdash-\hspace{0pt}}
\newcommand{\AG}[1]{\ensuremath{\mathrm{AG}(#1)}}
\newcommand{\si}[1]{\ensuremath{\operatorname{si}(#1)}}
\newcommand{\x}{$\bullet$}

\newcommand{\cat}[2]{\ensuremath{#1
\hspace{2.5pt plus 0.5pt minus 0.5pt}
\raisebox{-0.5pt}{\rotatebox{90}{$\Bowtie$}}
\hspace{2.5pt plus 0.5pt minus 0.5pt}#2}}

\theoremstyle{plain}
\newtheorem{theorem}{Theorem}[section]
\newtheorem{lemma}[theorem]{Lemma}
\newtheorem{corollary}[theorem]{Corollary}

\title[Binary matroids with no $M(K_{5}\ba e)$\dash minor]
{The internally $4$\protect\dash connected binary matroids with no
$M(K_{5}\ba e)$\dash minor}

\author{Dillon Mayhew\and Gordon Royle}

\begin{document}

\begin{abstract}
Let \cat{\AG{3,2}}{U_{1,1}} denote the binary matroid
obtained from $\AG{3,2}\oplus U_{1,1}$ by completing the
$3$\dash point lines between every element in $\AG{3,2}$ and
the element of $U_{1,1}$.
We prove that every internally $4$\dash connected binary matroid
that does not have a minor isomorphic to $M(K_{5}\ba e)$ is
isomorphic to a minor of $(\cat{\AG{3,2}}{U_{1,1}})^{*}$.
\end{abstract}

\maketitle

\section{Introduction}

The study of minor-closed classes of graphs has a long history, starting with Wagner's characterization~\cite{Wag37} of planar graphs as those graphs with no $K_5$ or $K_{3,3}$ minor, and Hall's study~\cite{Hal43}
of the minor-closed class that arises by
excluding $K_{3,3}$. 

Numerous authors have studied minor-closed classes of graphs defined by some natural property, such as embeddability on a particular surface, aiming to determine the excluded minors for that class of graphs; the standard example (other than planarity) being the determination of the 35 excluded minors for the class of projective planar graphs (Archdeacon~\cite{Arc81}).  Other authors have studied the ``dual problem'' of finding a structural description of the minor-closed classes of graphs obtained by excluding a particular minor or set of minors. Examples of classes investigated in this way include those obtained by excluding $K_5\backslash e$, $K_5$, $K_{3,3}$, $V_8$, the cube, or the octahedron (see \cite{KM07} and the references therein).

Analogous questions arise in the study of matroids, particularly binary matroids, and a number of excluded minor theorems and structural characterisations of classes of matroids defined by excluded minors are known. For example, Oxley has determined the binary matroids with no
$4$\dash wheel minor, the ternary matroids with no $M(K_4)$ minor, and the regular matroids with no 5-wheel minor, among others (see \cite{Oxl87b, Oxl87c, Oxl89}). 
For graphs, the Kuratowski graphs $K_{3,3}$ and $K_5$ are especially important, and for binary matroids, the Kuratowski graphs and their duals play an analogous role (see Kung~\cite{Kun87} for an early study of these classes). For graphs, decomposition theorems exist for all classes produced by excluding either or both of the Kuratowski graphs, but for binary matroids our knowledge is incomplete. 
There are $15$ classes of binary matroids obtained by
excluding some non-empty subset of
$$\{M(K_{3,3}),M(K_{5}),M^{*}(K_{3,3}),M^{*}(K_{5})\}.$$
Mayhew, Royle and Whittle \cite{MRW10}  determined the internally 4-connected binary matroids with no $M(K_{3,3})$-minor, and
this leads directly to characterisations (and associated decomposition theorems) for 12 of these classes~\cite{MRWb}. This subsumes 
earlier results of Qin and Zhou \cite{QZ04}, who described the binary matroids obtained by excluding all the cycle and bond matroids of the Kuratowski graphs.

Our failure to characterize the remaining three classes
is due to the fact that we do not have a decomposition
theorem for the class of binary matroids with no
$M(K_{5})$\dash minor.  Unfortunately, while characterizing the internally $4$\dash connected binary matroids
with no $M(K_{3,3})$\dash minor is difficult, and requires
a long and technical proof, it seems that it
will be even more challenging to characterize binary matroids with no $M(K_{5})$\dash minor.

%
%
%
%
%

In this article, we prove a partial result by determining the
internally $4$\dash connected binary matroids with no
$M(K_{5}\ba e)$\dash minor.
Our hope is that this will assist us in future
exploration, since we now know that an internally $4$\dash connected
binary matroid with no $M(K_{5})$\dash minor either has
an $M(K_{5}\ba e)$\dash minor, or is one of the finite number of
matroids we describe in this article.
As in the characterization of the binary matroids with no
$M(K_{3,3})$\dash minor, the proof involves
a great deal of case-checking, enough so that
completing the argument by hand is not feasible.
Instead, we use a computer verification for these portions of the proof; however we emphasize that the computations, which were independently performed by two different means, are only used to verify routine assertions about specific matroids.

Robertson \& Seymour \cite{RS84} considered the {\em graphs} with no
$M(K_{5}\ba e)$\dash minor, obtaining the result that a 3\dash connected graph with no  $M(K_{5}\ba e)$\dash minor is either a wheel, isomorphic to $K_{3,3}$, or isomorphic to the triangular prism $(K_5\ba e)^*$.
(Figure~\ref{fig4} shows the graphs $K_{5} \ba e$ and the triangular prism.)

For a number of reasons, we find it more convenient to
consider binary matroids without an $M^{*}(K_{5}\ba e)$\dash minor.
Therefore we will henceforth state all our results in terms
of binary matroids with no minor isomorphic to the cycle
matroid of the triangular prism.
(We abbreviate this to ``no prism-minor" or ``prism-free''.)

\begin{figure}[htb]
\includegraphics{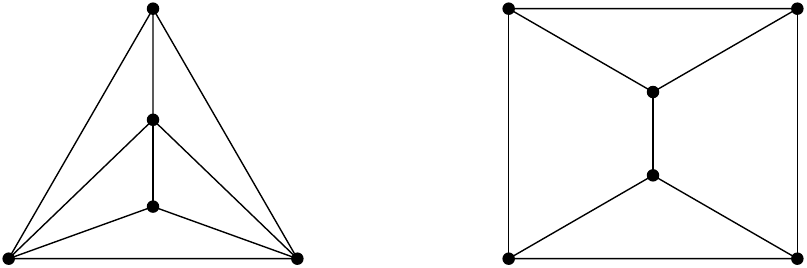}
\caption{$K_{5}\ba e$, and its geometric dual, the triangular prism.}
\label{fig4}
\end{figure}

Suppose that $M_{1}$ and $M_{2}$ are binary matroids on disjoint
ground sets.
Let $S$ be a set of cardinality
$|E(M_{1})|\times|E(M_{2})|$, disjoint from $E(M_{1})\cup E(M_{2})$,
where every element $e_{a,b}\in S$ corresponds to a
pair $(a,b)\in E(M_{1})\times E(M_{2})$.
Then \cat{M_{1}}{M_{2}} is the unique binary matroid
on the ground set $E(M_{1})\cup E(M_{2})\cup S$
satisfying
\[
(\cat{M_{1}}{M_{2}})|(E(M_{1})\cup E(M_{2})) = M_{1}\oplus M_{2}
\]
such that $\{a,b,e_{a,b}\}$ is a circuit, for every
$(a,b)\in E(M_{1})\times E(M_{2})$.
That is, $\cat{M_{1}}{M_{2}}$ is obtained from $M_{1}\oplus M_{2}$
by placing $e_{a,b}$ on the line between $a$ and $b$, for
every pair $(a,b)$ in $E(M_{1})\times E(M_{2})$.

\begin{theorem}
\label{1}
Let $M$ be a binary matroid with no prism-minor.
\begin{enumerate}[label={\rm (\roman*)}]
\item If $M$ is internally $4$\dash connected, then
$M$ has rank at most~$5$, and is isomorphic to a minor of
$\cat{\AG{3,2}}{U_{1,1}}$.
\item If $M$ is $3$\dash connected but not
internally $4$\dash connected, and $M$ has an internally
$4$\dash connected minor with at least~$6$ elements that
is not isomorphic to $M(K_{4})$, $F_{7}$, $F_{7}^{*}$,
or $M(K_{3,3})$, then $M$ is isomorphic to 
one of the sporadic matroids $\mathbf{S1}$, $\mathbf{S2}$,
$\mathbf{S3}$, $\mathbf{S4}$, or $\mathbf{S5}$,
defined in Table~{\rm \ref{tab:sporadic}} of
Section~{\rm \ref{sporadic}}.
\end{enumerate} 
\end{theorem}

Theorem~\ref{1} has a
peculiar consequence (Corollary~\ref{cor1}).
Apart from a finite number
of exceptions, every $3$\dash connected
binary matroid with no prism-minor can be constructed using
$3$\dash sums starting from copies of only two matroids:
$M(K_{4})$ and $F_{7}$.

Recall that a \emph{parallel extension} of the matroid $M$
is a matroid $M'$ with an element $e\in E(M')$ such that
$M'\ba e=M$, and $e$ is in a parallel pair of $M'$.
A \emph{cycle} of a binary matroid is a (possibly empty)
disjoint union of circuits.
If $M_{1}$ and $M_{2}$ are binary matroids such that
(i)~$|E(M_{1})|,|E(M_{2})| \geq 7$,
(ii)~$E(M_{1})\cap E(M_{2})=T$, where $T$ is a triangle
of both $M_{1}$ and $M_{2}$,
(iii)~$T$ does not contain a cocircuit in either
$M_{1}$ or $M_{2}$,
then the \emph{$3$\dash sum} of $M_{1}$ and $M_{2}$
(denoted $M_{1}\oplus_{3} M_{2}$)
is defined.
It is a binary matroid on the set $(E(M_{1})\cup E(M_{2}))-T$,
and the cycles of $M_{1}\oplus_{3} M_{2}$ are exactly
the sets of the form $(Z_{1}-Z_{2})\cup (Z_{2}-Z_{1})$,
where $Z_{i}$ is a cycle of $M_{i}$ for $i=1,2$, and
$Z_{1}\cap T=Z_{2}\cap T$.

\begin{corollary}
\label{cor1}
\begin{enumerate}[label={\rm (\roman*)}]
\item If $M$ is a $3$\dash connected binary matroid
with no prism-minor, then either
$M$ is an internally $4$\dash connected minor of
\cat{\AG{3,2}}{U_{1,1}}, or
$M$ is one of $\mathbf{S1}$, $\mathbf{S2}$,
$\mathbf{S3}$, $\mathbf{S4}$, or $\mathbf{S5}$,
or
$M$ can be constructed from copies of $M(K_{4})$ and $F_{7}$ using
parallel extensions and $3$\dash sums.
\item If $M$ is a $3$\dash connected binary matroid with no
$M(K_{5}\ba e)$\dash minor, then either
$M$ is an internally $4$\dash connected minor of
$(\cat{\AG{3,2}}{U_{1,1}})^{*}$, or
$M$ is one of
$\mathbf{S1}^{*}$, $\mathbf{S2}^{*}$,
$\mathbf{S3}^{*}$, $\mathbf{S4}^{*}$, or $\mathbf{S5}^{*}$, or
$M$ can be constructed from copies of $M(K_{4})$, $F_{7}$, and
$M^{*}(K_{3,3})$ using parallel extensions and $3$\dash sums.
\end{enumerate}
\end{corollary}

\begin{proof}
Suppose that $M$ is a counterexample to~(i),
chosen so that $|E(M)|$ is as small as possible.
Theorem~\ref{1}(i) means that $M$ is not
internally $4$\dash connected.
Therefore $M=M_{1}\oplus_{3} M_{2}$ for some matroids
$M_{1}$ and $M_{2}$ (see~\cite[(2.9)]{Sey80}).
Both $M_{1}$ and $M_{2}$ are minors of
$M$~\cite[(4.1)]{Sey80}, so neither $M_{1}$ nor
$M_{2}$ has a prism-minor.
Moreover, both $M_{1}$ and $M_{2}$ have fewer elements than $M$,
and \si{M_{1}} and \si{M_{2}} are both
$3$\dash connected~\cite[(4.3)]{Sey80},
so \si{M_{1}} and \si{M_{2}} satisfy
Corollary~\ref{cor1}(i).

Assume that \si{M_{1}} is one of
$\mathbf{S1}$, $\mathbf{S2}$,
$\mathbf{S3}$, $\mathbf{S4}$, or $\mathbf{S5}$.
In Section~\ref{sporadic} we certify that each of these
$5$ matroids contains an internally $4$\dash connected minor
$N$ such that $|E(N)|\geq 6$ and $N$ is not isomorphic
to $M(K_{4})$, $F_{7}$,
$F_{7}^{*}$, or $M(K_{3,3})$.
Thus \si{M_{1}}, and hence $M$, contains
an internally $4$\dash connected minor other
than $M(K_{4})$, $F_{7}$, $F_{7}^{*}$, or $M(K_{3,3})$.
This is a contradiction to Theorem~\ref{1}(ii), as $M$ is not
one of
$\mathbf{S1}$, $\mathbf{S2}$,
$\mathbf{S3}$, $\mathbf{S4}$, or $\mathbf{S5}$.
Therefore \si{M_{1}} is not one of the $5$ sporadic matroids.

Assume that \si{M_{1}} is internally $4$\dash connected.
As $M_{1}$ is a term in a $3$\dash sum, it contains a
triangle $T$, and $T$ does not contain a cocircuit of
$M_{1}$.
This means that $r(M_{1})\geq 3$.
Since \si{M_{1}} is $3$\dash connected,
it follows that \si{M_{1}} has at least~$6$ elements.
Now $M$ contains \si{M_{1}} as a minor, and $M$ is not
one of the sporadic matroids.
Theorem~\ref{1}(ii) implies that \si{M_{1}} is
isomorphic to $M(K_{4})$, $F_{7}$, $F_{7}^{*}$, or $M(K_{3,3})$.
But $M_{1}$ contains a triangle,
so \si{M_{1}} is not isomorphic to $F_{7}^{*}$ or $M(K_{3,3})$.
Hence \si{M_{1}} is isomorphic to $M(K_{4})$ or
$F_{7}$.

On the other hand, if \si{M_{1}} is not internally $4$\dash connected,
then by the inductive hypothesis, $M_{1}$ can be constructed from
copies of $M(K_{4})$ and $F_{7}$ using
parallel extensions and $3$\dash sums.
Thus, in either case, $M_{1}$
(and $M_{2}$, by an identical argument), can be
constructed from
copies of $M(K_{4})$ and $F_{7}$ using
parallel extensions and $3$\dash sums.
Therefore the same statement holds for $M$.
This contradiction completes the proof of~(i).
The proof of~(ii) is almost identical:
if $M$ is a minimal counterexample, then $M$ can be expressed as
$M_{1}\oplus_{3} M_{2}$, and if \si{M_{1}} is internally
$4$\dash connected, then it must be isomorphic to
$M(K_{4})$, $F_{7}$, $F_{7}^{*}$, or $M^{*}(K_{3,3})$.
But $\si{M_{1}}\ncong F_{7}^{*}$, as \si{M_{1}} must contain a
triangle.
\end{proof}

Our main theoretical tool in the proof of Theorem~\ref{1} is a
chain theorem for internally $4$\dash connected binary
matroids~\cite{CMO}. This theorem tells us that if a counterexample to Theorem~\ref{1}
exists, then it has at most three more elements than one of the known internally $4$\dash connected prism-free binary matroids, and therefore has rank at most~$8$.
Thus the proof of the theorem is reduced to a
(large) finite case analysis to demonstrate that none of the
internally $4$\dash connected prism-free binary matroids of rank up to 5 can
be extended and/or coextended by up to three elements to form a new internally $4$\dash connected prism-free binary matroid. This case analysis is performed by computer, but to increase
confidence in the correctness of the result, we conducted this analysis
in two totally independent ways.

The first search relies upon the \textsc{Macek} software
package by Petr Hlin\v{e}n\'{y} \cite{Hli04}. 
Macek has the facility to extend or coextend matroids while avoiding specified minors and using this it is relatively easy to verify that there are no new internally $4$\dash connected prism-free binary matroids within three extension/coextension steps of any of the known ones. The second technique involves the construction from first principles of a 
database of {\em all} simple prism-free binary matroids of rank up to 8.
Having constructed such a database, verifying that
Theorem~\ref{1} holds merely requires checking that none of the rank 6, 7 or 8 prism-free binary matroids are internally $4$\dash connected.  Both these computer searches are described in more detail in Section~\ref{compsearch}.

\section{Listing matroids}

In this section we explicitly list the 42 internally $4$\dash connected prism-free binary
matroids, and the five sporadic matroids that appear in the statement of Theorem~\ref{1}.

\subsection{The internally $4$-connected prism-free binary matroids}
\label{listsection}

Let $\mathcal{M}$ be the set of internally $4$\dash connected
minors of \cat{\AG{3,2}}{U_{1,1}} which, by Theorem~\ref{1}, is 
exactly the set of internally $4$\dash connected binary 
matroids with no prism-minors. 

There are 42 matroids in $\mathcal{M}$. 
It is easy to see that the only internally $4$\dash connected
binary matroids on at most $5$ elements are
$U_{0,0}$, $U_{0,1}$, $U_{1,1}$, $U_{1,2}$, $U_{1,3}$, and
$U_{2,3}$, so from this point
we content ourselves with listing the $36$ matroids in
$\mathcal{M}$ that have at least~$6$ elements.

The only rank\dash $3$ members
of $\mathcal{M}$ are $M(K_{4})$ and $F_{7}$,
which (for the sake of consistency), we will denote with
{\bf M1} and {\bf M2}, respectively.

Let $e$ be an element of \AG{3,2}.
It is an easy exercise to show that
contracting $e$ from \cat{\AG{3,2}}{U_{1,1}} and then
simplifying produces a rank\dash $4$ binary matroid with~$15$ elements,
which is therefore isomorphic to $\mathrm{PG}(3,2)$.
Hence $\mathcal{M}$ contains every rank\dash $4$
internally $4$\dash connected binary matroid.
Every such matroid can be obtained by deleting columns
from the matrix representing  $\mathrm{PG}(3,2)$ which is shown at the 
head of Table~\ref{tab2}.
Each row of the table corresponds to an internally $4$\dash connected
binary matroid with rank~$4$.
The empty entries in that row correspond to columns which
should be deleted to obtain a representation.
For example, {\bf M9} is represented by the matrix
\setcounter{MaxMatrixCols}{15}
$$
\begin{bmatrix}
1&0&&&1&1&1&0&0&&1&1&1&0&1\\
0&1&&&1&0&0&1&1&&1&1&0&1&1\\
0&0&&&0&1&0&1&0&&1&0&1&1&1\\
0&0&&&0&0&1&0&1&&0&1&1&1&1\\
\end{bmatrix},
$$
which has been obtained by deleting three particular columns from the matrix 
representing $\mathrm{PG}(3,2)$.

\begin{table}[htb]
\begin{center}
\begin{tabular}{|r|cccc|cccccc|cccc|c|l|}
\hline
&1&0&0&0&1&1&1&0&0&0&1&1&1&0&1&\\
&0&1&0&0&1&0&0&1&1&0&1&1&0&1&1&\\
&0&0&1&0&0&1&0&1&0&1&1&0&1&1&1&\\
&0&0&0&1&0&0&1&0&1&1&0&1&1&1&1&\\
\hline
{\bf M3}&&&&&\x&\x&\x&&&&\x&\x&\x&&\x& $F_7^*$\\
\hline
{\bf M4} &&&&&&\x&\x&\x&\x&&\x&\x&\x&\x&\x&$M^*(K_{3,3})$\\
\hline
{\bf M5}&&&&&\x&\x&\x&\x&\x&\x&\x&\x&\x&\x&&$M(K_5)$\\
{\bf M6}&&&&&\x&\x&\x&\x&\x&&\x&\x&\x&\x&\x&\\
\hline
{\bf M7}&&\x&&&\x&\x&\x&\x&\x&&\x&\x&\x&\x&\x&\\
{\bf M8}&&&&&\x&\x&\x&\x&\x&\x&\x&\x&\x&\x&\x&\\
\hline
{\bf M9}&\x&\x&&&\x&\x&\x&\x&\x&&\x&\x&\x&\x&\x&\\
{\bf M10}&\x&&&&\x&\x&\x&\x&\x&\x&\x&\x&\x&\x&\x&\\
\hline
{\bf M11}&\x&\x&&&\x&\x&\x&\x&\x&\x&\x&\x&\x&\x&\x&\\
\hline
{\bf M12}&\x&\x&\x&&\x&\x&\x&\x&\x&\x&\x&\x&\x&\x&\x&\\
\hline
{\bf M13}&\x&\x&\x&\x&\x&\x&\x&\x&\x&\x&\x&\x&\x&\x&\x&$\mathrm{PG}(3,2)$\\
\hline
\end{tabular}
\end{center}
\caption{Internally $4$\protect\dash connected rank\protect\dash $4$
restrictions of $\mathrm{PG}(3,2)$.}
\label{tab2}
\end{table}

We use a similar format to describe the rank\dash $5$ members of
$\mathcal{M}$.
Let $A$ be the matrix
$$
\begin{bmatrix}
0&1&1&1\\
1&0&1&1\\
1&1&0&1\\
1&1&1&0\\
\end{bmatrix}.
$$
Then $[I_{4}|A]$ is a $\mathrm{GF}(2)$\dash representation
of $\AG{3,2}$.
It follows that
$\cat{\AG{3,2}}{U_{1,1}}$ is represented over
$\mathrm{GF}(2)$ by the matrix
\begin{equation*}
\left[
\begin{array}{c|c|c|c|c}
\mathbf{0} & I_{4} & A & I_{4} & A\\[0.5ex]\hline
\rule{0ex}{3ex}1 & \mathbf{0}^{T} & \mathbf{0}^{T} & \mathbf{1}^{T} & \mathbf{1}^{T}\\
\end{array}
\right],
\end{equation*}
where $\mathbf{0}$ (respectively $\mathbf{1}$) is the
$4\times 1$ vector of all zeros (respectively ones).

A representation of any rank\dash $5$ member of $\mathcal{M}$
can be obtained by deleting columns from this matrix and each 
row of Table~\ref{tab1} corresponds to a rank\dash $5$ member of $\mathcal{M}$. We note here that {\bf M14} is $M(K_{3,3})$, and
{\bf M16} is the regular matroid $R_{10}$.

\begin{table}[htb]
\begin{center}
\begin{tabular}{|r|c|cccc|cccc|cccc|cccc|}
\hline
&0&1&0&0&0&0&1&1&1&1&0&0&0&0&1&1&1\\
&0&0&1&0&0&1&0&1&1&0&1&0&0&1&0&1&1\\
&0&0&0&1&0&1&1&0&1&0&0&1&0&1&1&0&1\\
&0&0&0&0&1&1&1&1&0&0&0&0&1&1&1&1&0\\
&1&0&0&0&0&0&0&0&0&1&1&1&1&1&1&1&1\\
\hline
{\bf M14}&&\x&\x&\x&\x&\x&\x&&&\x&\x&\x&&&&&\\
\hline
{\bf M15}&\x&\x&\x&\x&\x&\x&\x&&&\x&\x&&&&&\x&\\
{\bf M16}&&\x&\x&\x&\x&\x&\x&&&\x&\x&\x&&&&\x&\\
\hline
{\bf M17}&\x&\x&\x&\x&\x&\x&\x&\x&&\x&\x&&\x&&&&\\
 {\bf M18}&\x&\x&\x&\x&\x&\x&\x&&&\x&\x&\x&&&&\x&\\
{\bf M19}&&\x&\x&\x&\x&\x&\x&\x&&\x&\x&\x&\x&&&&\\
\hline
 {\bf M20}&\x&\x&\x&\x&\x&\x&\x&\x&&\x&\x&&\x&&&\x&\\
{\bf M21}&\x&\x&\x&\x&\x&\x&\x&\x&&\x&\x&&\x&&&&\x\\
{\bf M22}&\x&\x&\x&\x&\x&\x&\x&\x&&\x&\x&\x&\x&&&&\\
 {\bf M23}&&\x&\x&\x&\x&\x&\x&\x&\x&\x&\x&\x&\x&&&&\\
\hline
{\bf M24}&\x&\x&\x&\x&\x&\x&\x&\x&&\x&\x&&\x&&&\x&\x\\
{\bf M25}&\x&\x&\x&\x&\x&\x&\x&\x&\x&\x&\x&\x&\x&&&&\\
{\bf M26}&\x&\x&\x&\x&\x&\x&\x&\x&&\x&\x&\x&\x&\x&&&\\
{\bf M27}&&\x&\x&\x&\x&\x&\x&\x&\x&\x&\x&\x&\x&\x&&&\\
\hline
 {\bf M28}&\x&\x&\x&\x&\x&\x&\x&\x&\x&\x&\x&\x&\x&\x&&&\\
{\bf M29}&\x&\x&\x&\x&\x&\x&\x&\x&&\x&\x&\x&\x&\x&\x&&\\
{\bf M30}&&\x&\x&\x&\x&\x&\x&\x&\x&\x&\x&\x&\x&\x&\x&&\\
\hline
{\bf M31}&\x&\x&\x&\x&\x&\x&\x&\x&\x&\x&\x&\x&\x&\x&\x&&\\
{\bf M32}&\x&\x&\x&\x&\x&\x&\x&\x&&\x&\x&\x&\x&\x&\x&\x&\\
 {\bf M33}&&\x&\x&\x&\x&\x&\x&\x&\x&\x&\x&\x&\x&\x&\x&\x&\\
\hline
{\bf M34}&\x&\x&\x&\x&\x&\x&\x&\x&\x&\x&\x&\x&\x&\x&\x&\x&\\
 {\bf M35}&&\x&\x&\x&\x&\x&\x&\x&\x&\x&\x&\x&\x&\x&\x&\x&\x\\
\hline
 {\bf M36}&\x&\x&\x&\x&\x&\x&\x&\x&\x&\x&\x&\x&\x&\x&\x&\x&\x\\
\hline
\end{tabular}
\end{center}
\caption{Internally $4$\protect\dash connected rank\protect\dash $5$
restrictions of
\protect\cat{\AG{3,2}}{U_{1,1}}.}
\label{tab1}
\end{table}

\subsection{Sporadic matroids.}
\label{sporadic}
In this section we describe the $5$ sporadic matroids that
appear in the statement of Theorem~\ref{1}.
For the sake of brevity we represent an $n$\dash element binary matroid
$M$ with a string of numbers $m_{1},\ldots, m_{n}$ between $1$ and
$2^{r(M)}-1$.
A representation of $M$ over $\mathrm{GF}(2)$ can be obtained
by taking the binary representations of
$m_{1},\ldots, m_{n}$, each of which has $r(M)$ bits, and
taking these binary representations to be the columns of a
matrix $A$.
We use the convention that the least significant bit of
the binary representation will be in the bottom row
of $A$, and the most significant will be in the top row.
Thus the sequence
$$
1,2,3,4,6,8,9,12,
$$
corresponds to the following matrix:
$$A = 
\begin{bmatrix}
0&0&0&0&0&1&1&1\\
0&0&0&1&1&0&0&1\\
0&1&1&0&1&0&0&0\\
1&0&1&0&0&0&1&0\\
\end{bmatrix}.
$$
We also use the numbers
$m_{1},\ldots, m_{n}$ to represent the corresponding elements
in $E(M)$. With this notational convention, the $5$ sporadic matroids are presented in Table~\ref{tab:sporadic}.

\begin{table}[htb]
\begin{tabular}{| c | c | c | l |}
\hline
Name&Size&Rank&Elements\\
\hline
$\mathbf{S1}$ &11&5&$1,4,5,8,9,14,15,16,22,27,29$\\
$\mathbf{S2}$ &11&5&$1,3,4,8,9,14,15,16,22,27,29$\\
$\mathbf{S3}$ &12&5&$1,3,4,5,8,9,14,15,16,22,27,29$\\
$\mathbf{S4}$ &12&6&$1,2,4,8,15,16,32,42,44,49,56,63$\\
$\mathbf{S5}$ &13&5&$1,2,3,4,5,8,9,14,15,16,22,27,29$\\
\hline
\end{tabular}
\caption{The sporadic $3$\protect\dash connected prism-free binary matroids.}
\label{tab:sporadic}
\end{table}

By deleting or contracting at most two elements from
any of these sporadic matroids, we can obtain
one of the internally $4$\dash connected matroids listed
in Section~\ref{listsection}, as we now certify.
\begin{itemize}
\item Contracting~$16$ from {\bf S1} produces {\bf M5}.
\item Contracting~$16$ from {\bf S2} produces {\bf M6}.
\item Deleting~$3$ or $5$ from {\bf S3} produces
{\bf S1} or {\bf S2} respectively.
\item Contracting~$1$ from {\bf S4} produces
{\bf S1}.
\item Deleting $29$ from {\bf S5} produces
{\bf M20}.

\end{itemize}

\section{Applying a chain theorem}

We reduce the proof of Theorem~\ref{1} to a finite case check by using a chain theorem for internally $4$\dash connected binary matroids, but  before stating this theorem, we define the families of graphs appearing in it.

For $n \ge 3$, the \emph{planar quartic ladder} on $2n$ vertices
consists of two disjoint cycles
$$
\{u_{0}u_{1}, u_{1}u_{2}, \ldots, u_{n-2}u_{n-1}, u_{n-1}u_0\} \cup
\{v_{0}v_{1},v_{1}v_{2},\ldots, v_{n-2}v_{n-1}, v_{n-1}v_0\}$$
and two
perfect matchings
$$
\{u_{0}v_{0}, u_{1}v_{1}, \ldots, u_{n-1}v_{n-1}\} \cup
\{u_{0}v_{n-1}, u_{1}v_{0}, \ldots, u_{n-1}v_{n-2}\}.$$
Each planar quartic ladder contains all smaller planar quartic ladders as
minors, and the smallest planar quartic ladder is the 
\emph{octahedron} on~$6$ vertices. For $n \ge 3$, the \emph{M\"{o}bius quartic ladder} on $2n-1$ vertices
consists of a Hamilton cycle
$$\{v_{0}v_{1},v_{1}v_{2},\ldots, v_{2n-3}v_{2n-2},v_{2n-2}v_{0}\}$$
and the set of edges
$$\{v_{i}v_{i+n-1},v_{i}v_{i+n}\mid 0\leq i\leq n-1\},$$ where
subscripts are read modulo~$2n-1$. Each quartic M\"obius ladder contains all smaller quartic M\"obius ladders as minors, and the smallest quartic 
M\"obius ladder is the complete graph $K_5$ on~$5$ vertices. Figure~\ref{fig1} shows one member of each of these families.

Finally, the \emph{terrahawk} is obtained from the cube by adding a new
vertex, and making it adjacent to the four vertices in a
face of the cube. Figure~\ref{fig2} shows diagrams of the cube, the octahedron,
and the terrahawk. With these definitions in hand, we can state the chain theorem:
\newcommand{\ifourc}{internally $4$\dash connected }
\begin{theorem}[Chun, Mayhew, Oxley \cite{CMO}]
\label{2}
Let $M$ be an \ifourc binary matroid such that $|E(M)|\geq 7$.
Then $M$ has a proper \ifourc minor $N$ with $|E(M)|-|E(N)| \le 3$ unless $M$ or its dual is the cycle matroid of a planar quartic ladder or M\"obius quartic ladder, or a terrahawk. 
\end{theorem}

\begin{figure}[htb]
\includegraphics{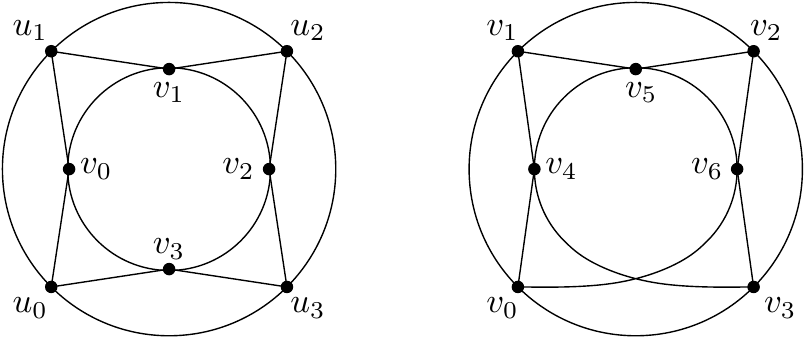}
\caption{A planar quartic ladder, and a M\"{o}bius quartic ladder.}
\label{fig1}
\end{figure}

\begin{figure}[htb]
\includegraphics{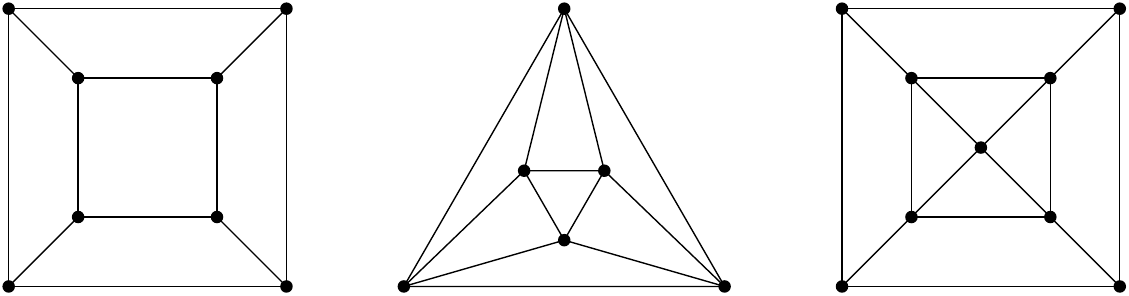}
\caption{The cube, the octahedron, and the terrahawk.}
\label{fig2}
\end{figure}

The next lemma shows that the exceptional cases in Theorem~\ref{2}
all have prism minors except for $M(K_5)$.

\begin{lemma}
\label{3}
The cycle and bond matroids of
(i) the terrahawk, (ii) the planar quartic ladders, and (iii) the M\"obius quartic ladders with at least~$7$ vertices, all have the triangular prism as a minor.
Moreover, the bond matroid of the M\"{o}bius quartic ladder
with~$5$ vertices has the triangular prism as a minor.
\end{lemma}

\begin{proof}
The cycle matroid of the terrahawk, which is self-dual, clearly has a cube minor, which itself clearly has a triangular prism minor. The planar quartic ladders all have an octahedron minor, while their duals all have a cube minor, and both the octahedron and the cube have a triangular prism minor. 
A M\"obius quartic ladder on at least seven vertices contains
the M\"obius quartic ladder with seven vertices, which in
turn contains a triangular prism minor (see Figure~\ref{fig3}).
Since the smallest M\"obius quartic ladder is isomorphic to
$K_{5}$, the bond matroid of a M\"obius quartic ladder contains
$M^*(K_5)$, and hence $M^*(K_5 \ba e)$. 
\end{proof}

\begin{figure}[htb]
\includegraphics{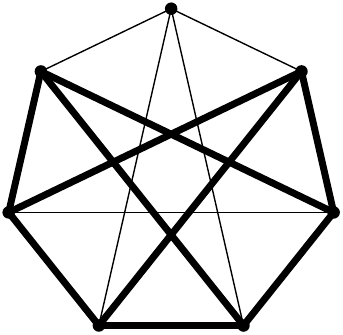}
\caption{The M\"{o}bius quartic ladder on seven vertices.}
\label{fig3}
\end{figure}

Assume that $M$ is a minimal counterexample to
Theorem~\ref{1}(i).
Then $M$ is internally $4$\dash connected, and
$|E(M)|\geq 7$, as otherwise $M$ is certainly a minor
of \cat{\AG{3,2}}{U_{1,1}}.
Note that $M(K_{5})$ is a minor of
\cat{\AG{3,2}}{U_{1,1}}, so $M\ncong M(K_{5})$.
This fact, and Lemma~\ref{3}, implies $M$ is not
the cycle or bond matroid of a quartic ladder, or of
the terrahawk.
Thus Theorem~\ref{2} says that
$M$ contains an internally
$4$\dash connected minor $N$ satisfying
$1\leq |E(M)|-|E(N)|\leq 3$.
The minimality of $M$ implies that $N$ is one
of the internally $4$\dash connected matroids
listed in Section~\ref{listsection}.
Therefore, if a counterexample exists, we can find it
by extending and coextending by at most three elements,
starting with the known 
internally $4$\dash connected matroids.
Next we show that 
the extensions and/or coextensions required to find $M$ starting from $N$
can be chosen to maintain $3$\dash connectivity at each step.

Let $n\geq 3$ be an integer.
Recall that the \emph{wheel} graph $\mathcal{W}_{n}$ is obtained
from a cycle with $n$ vertices by adding a new vertex and joining
it to each of the $n$ vertices in the cycle.

\begin{lemma}
\label{6}
Let $M$ be an internally $4$\dash connected binary matroid such
that $|E(M)|\geq 7$ and $M$ does not have a prism-minor.
If $M$ is not isomorphic to $M(K_{5})$, $M(K_{3,3})$ or
$M^{*}(K_{3,3})$, then there is a sequence $M_{0},\ldots, M_{t}$
of $3$\dash connected matroids such that:
\begin{enumerate}[label=(\roman*)]
\item $M_{0}$ is internally $4$\dash connected,
\item $M_{t}=M$,
\item $1\leq t\leq 3$, and
\item $M_{i+1}$ is a single-element extension or coextension
of $M_{i}$, for every $i\in\{0,\ldots, t-1\}$.
\end{enumerate}
\end{lemma}

\begin{proof}
By Theorem~\ref{2} and Lemma~\ref{3}, $M$ contains an internally $4$\dash connected
minor $N$ such that $1\leq |E(M)|-|E(N)|\leq 3$.
We let $M_{0}$ be equal to $N$.
If $N$ is not a wheel, then the result follows immediately
from Seymour's Splitter Theorem~(see~\cite[Theorem~12.1.2]{Oxl11}).
Therefore we must assume that $N$ is isomorphic to
$M(\mathcal{W}_{n})$ for some $n$.
If $n \geq 4$, then $M(\mathcal{W}_{n})$ is not
internally $4$\dash connected, so
$N\cong M(\mathcal{W}_{3})\cong M(K_{4})$.
If $M$ does not have any larger wheel as a minor, then we can
again apply the Splitter Theorem and deduce that the
result holds.
Therefore we assume that $M$ contains $\mathcal{W}_{4}$ as a minor.
As $|E(\mathcal{W}_{4})|=8$, and $|E(M)| \leq |E(K_{4})|+3=9$,
and $M(\mathcal{W}_{4})$ is not internally $4$\dash connected,
we deduce that $M$ is a single-element extension or coextension
of $M(\mathcal{W}_{4})$.

Let us assume that
$M$ is an extension of $M(\mathcal{W}_{4})$.
Recall that $M(\mathcal{W}_{4})$ is represented over
$\mathrm{GF}(2)$ by the following matrix.
$$
\kbordermatrix{
&a&b&c&d&e&f&g&h\\
&1&0&0&0&1&0&0&1\\
&0&1&0&0&1&1&0&0\\
&0&0&1&0&0&1&1&0\\
&0&0&0&1&0&0&1&1\\
}
$$
The uniqueness of representations over $\mathrm{GF}(2)$
(see \cite[Proposition~6.6.5]{Oxl11})
means that a representation of $M$ is obtained
by adding a column to this matrix.
If this new column contains a zero in the first row, then
$\{a,e,h\}$ is a triad of $M$, and $\{a,b,e\}$ is a triangle.
This leads to a violation of internal $4$\dash connectivity.
Therefore the new column contains a one in the first row.
Repeating this argument shows that the new column must
contain ones everywhere.
Now it is easy to see that $M\cong M^{*}(K_{3,3})$,
contradicting the hypotheses of the lemma.
A dual argument shows that if $M$ is a coextension of
$M(\mathcal{W}_{4})$, then $M\cong M(K_{3,3})$, so we are done.
\end{proof}

\section{Computer searches}
\label{compsearch}

We use two independent computer searches to verify the following
lemmas, and thereby prove Theorem~\ref{1}. Section~\ref{macek} outlines
a technique using \textsc{Macek} to extend the known \ifourc prism-free binary matroids, while Section~\ref{orderly} describes an exhaustive search for all prism-free binary
matroids of rank up to 8. The results from both searches were in total agreement.

\begin{lemma}
\label{4}
Let $M$ be an internally $4$\dash connected
binary matroid with no prism-minor.
Then $M$ is one of the $42$ matroids described
in Section~{\rm \ref{listsection}}.
\end{lemma}

\begin{lemma}
\label{5}
Let $M$ be a binary matroid with no prism-minor.
If $M$ is $3$\dash connected but not internally $4$\dash connected,
and $M$ contains an internally $4$\dash connected minor with at least~$6$ elements that is not isomorphic to $M(K_{4})$, $F_{7}$,
$F_{7}^{*}$, or $M(K_{3,3})$, then $M$ is one of the
$5$ sporadic matroids
$\mathbf{S1}$, $\mathbf{S2}$, $\mathbf{S3}$, $\mathbf{S4}$,
or $\mathbf{S5}$.
\end{lemma}

\subsection{Using MACEK}\label{macek}

It is straightforward to use \textsc{Macek} to find all
$3$\dash connected binary matroids $M$ satisfying the
following properties:
\begin{enumerate}[label=(\roman*)]
\item $M$ does not have a prism-minor;
\item there is a sequence $M_{0},\ldots, M_{t}$ of
$3$\dash connected matroids where
\begin{enumerate}[label=(\alph*)]
\item $M_{0}$ has at least~$6$ elements, and is one of the
internally $4$\dash connected matroids listed in Section~\ref{listsection},
\item $M_{t}=M$,
\item $1\leq t\leq 3$,
\item $M_{i+1}$ is a single-element extension or
coextension of $M_{i}$, for all $i\in\{0,\ldots, t-1\}$;
\end{enumerate}
\item\label{item} $M$ does not contain a minor $N$ where
$|E(N)| = |E(M_{0})|+1$, and $N$ is isomorphic to one of the
internally $4$\dash connected matroids listed
in Section~\ref{listsection}.
\end{enumerate}

For example, if $M_{0}$ is {\bf M4}, then we would use the
command
\begin{verbatim}
./macek -pGF2 `@ext-forbid Prism M5 M6 M15 M16;!extend bbb' M4
\end{verbatim}
to find $M$,
since {\bf M5}, {\bf M6}, {\bf M15}, and
{\bf M16} are the matroids in Section~\ref{listsection}
that are a single element larger than {\bf M4}.

Assume that $M$ is a counterexample
to Lemma~\ref{4}, chosen so that $|E(M)|$ is as small
as possible.
Then $M$ certainly has an
$M(K_{4})$\dash minor~\cite[Corollary~12.2.13]{Oxl11}, but
$M$ is not isomorphic to $M(K_{4})$.
Therefore $|E(M)| \geq 7$.
Since $M(K_{5})$, $M(K_{3,3})$, and $M^{*}(K_{3,3})$
all appear in Section~\ref{listsection}, $M$ is isomorphic
to none of these matroids.
We consider the sequence of matroids 
$M_{0},\ldots, M_{t}$ supplied by Lemma~\ref{6}.
As there are no internally $4$\dash connected binary matroids
with four or five elements, it is certainly the case
that $|E(M_{0})| \geq 6$.
Moreover, $M_{0}$ is listed in Section~\ref{listsection},
by the minimality of $M$.
We assume that $M$ and $M_{0}$ have been chosen so that
$t$ is as small as possible.
This means that $M$ cannot contain a minor $N$ such that
$|E(N)|=|E(M_{0})|+1$ and $N$ is listed in
Section~\ref{listsection}, or else we would have chosen $M_{0}$
to be $N$ instead.
Thus condition~\ref{item} in the list above holds.
Therefore $M$ will be found in the \textsc{Macek} search we described at the
beginning of this section.
But the \textsc{Macek} search uncovers no internally $4$\dash connected matroids.
We conclude that Lemma~\ref{4} holds.

Now assume that $M$ is a minimal counterexample to Lemma~\ref{5}.
Then $M$ contains an internally $4$\dash connected minor
$M_{0}$ such that $|E(M_{0})| \geq 6$ and $M_{0}$ is not isomorphic
to $M(K_{4})$, $F_{7}$, $F_{7}^{*}$, or $M(K_{3,3})$.
Lemma~\ref{4} implies that $M_{0}$ is one of the matroids
listed in Section~\ref{listsection}.
By Seymour's Splitter Theorem, there is a sequence
of $3$\dash connected matroids $M_{0},\ldots, M_{t}$
such that $M_{t}=M$, and each matroid in the sequence
is a single-element extension or coextension of the previous
matroid.
We assume that $M$ and $M_{0}$ have been chosen so that
$|E(M)|-|E(M_{0})|=t$ is as small as possible.

Suppose that $M$ contains as a minor
an internally $4$\dash connected matroid with exactly
one more element than $M_{0}$.
By the minimality of $t$, this internally $4$\dash connected
matroid must be $M(K_{4})$, $F_{7}$, $F_{7}^{*}$, or $M(K_{3,3})$.
This implies $M_{0}$ has at most $8$ elements, and is therefore
isomorphic to $M(K_{4})$, $F_{7}$, or $F_{7}^{*}$.
This contradiction shows that $M$ and $M_{0}$ obey
condition~\ref{item}, described above.

Certainly $t\geq 1$.
If $M_{t-1}$ is internally $4$\dash connected, then
$t=1$.
Assume that $M_{t-1}$ is not internally $4$\dash connected.
By the minimality of $t$, $M_{t-1}$ is not a counterexample to
Lemma~\ref{5}, so
it must be one of the sporadic matroids.
Every such matroid has an internally $4$\dash connected minor
that is at most two elements smaller.
Moreover, none of these internally $4$\dash connected
matroids is isomorphic to
$M(K_{4})$, $F_{7}$, $F_{7}^{*}$, or $M(K_{3,3})$.
It follows from this that $M_{0}=M_{t-3}$ or $M_{0}=M_{t-2}$,
and therefore $t\leq 3$.
Therefore $M$ will be uncovered by the \textsc{Macek} procedure we described
earlier.
However, when we apply the \textsc{Macek} search procedure to the
internally $4$\dash connected matroids in
Section~\ref{listsection} other than
$M(K_{4})$, $F_{7}$, $F_{7}^{*}$, or $M(K_{3,3})$, we produce no
$3$\dash connected matroids other than the sporadic matroids.
We conclude that Lemma~\ref{5} is true.

\subsection{An exhaustive search process}\label{orderly}

In this section we describe the exhaustive search process that was used to compute the list of simple binary matroids with no
$M^{*}(K_{5}\ba e)$\dash minor of rank up to 8.
This is accomplished by mapping the problem into a graph-theoretic context and then using an orderly algorithm designed for graphs to perform the computations.

First, note that a simple binary matroid of rank at most $r$ can be identified with a set of points in the projective space
$\mathrm{PG}(r-1,2)$. Given the matroid $M$, we can take the columns of an arbitrary representing matrix, which are non-zero as $M$ is simple, to be vectors in $\mathrm{GF}(2)^r$ (padding them with zeros if $M$ has rank less than $r$). Conversely any set of points in $\mathrm{PG}(r-1,2)$ determines a matroid of rank at most $r$ simply by taking the unique non-zero vector representing each point to be the columns of a matrix. The importance of this identification is that the unique representability of binary matroids ensures that the natural concepts of equivalence in each context coincide. In particular, two simple binary matroids are isomorphic if and only if the two corresponding subsets of $\mathrm{PG}(r-1,2)$ are equivalent under the automorphism group of the projective space, which is the projective general linear group $\mathrm{PGL}(r,2)$. 

\begin{lemma}
Let $X$ and $Y$ denote sets of points in the projective space $\mathrm{PG}(r-1,2)$. Then the matroids determined by $X$ and $Y$ are isomorphic if and only if there is an element of $\mathrm{PGL}(r,2)$ mapping $X$ to $Y$. \qed
\end{lemma}

To express this in a graph-theoretic context, we need to work with a graph whose automorphism group is $\mathrm{PGL}(r,2)$. So let $\Gamma_r$ denote the point-hyperplane incidence graph of $\mathrm{PG}(r-1,2)$; this is a bipartite graph with $2 (2^r-1)$ vertices of which $2^r-1$ are ``point-type'' vertices and $2^r-1$ are ``hyperplane-type'' vertices. The automorphism group of this graph is $\mathrm{PGL}(r,2) \times 2$  where the extra factor of 2 arises from an automorphism that exchanges points with hyperplanes. Let ${\mathcal P}_r$ denote the ``point-type'' vertices of $\Gamma_r$. Thus any simple binary matroid $M$ of rank at most $r$ can be identified with a subset ${\mathcal P}_r(M)$ of ${\mathcal P}_r$, and the automorphism group of $\Gamma_r$ fixing ${\mathcal P}_r(M)$ is the automorphism group of $M$.

Brendan McKay's graph isomorphism and canonical labelling program \verb+nauty+ can find the automorphism group and canonical labelling of a graph with a given set of vertices distinguished (i.e., a coloured graph). Therefore two simple binary matroids $M$ and $N$ of rank at most $r$ are isomorphic if and only if $|M|=|N|$ and the canonically labelled isomorph of $\Gamma_r$ with ${\mathcal P}_r(M)$ distinguished is identical to the canonically labelled isomorph of $\Gamma_r$ with ${\mathcal P}_r(N)$ distinguished.

\subsubsection{An orderly algorithm}
Suppose first that our task is to compute {\em all} the simple binary matroids up to some fixed rank $r$ ---  a simple modification to this basic algorithm will permit the computation of simple binary matroids excluding any particular matroid or set of matroids. 

As in the previous section, let $\Gamma_r$ be the bipartite point-hyperplane graph of $\mathrm{PG}(r-1,2)$, let ${\mathcal P}_r$ denote the points of $\mathrm{PG}(r-1,2)$ and let $G = \mathrm{PGL}(r,2)$ be the subgroup of $\mathrm{Aut}(\Gamma_{r})$ that fixes ${\mathcal P}_r$. 
Then our aim is the following:
\begin{quote}
{\em Compute one representative of each $G$\dash orbit on subsets of $\mathcal{P}_r$.}
\end{quote}

Let ${\mathcal L}_k$ be a set containing one representative from each $G$\dash orbit on $k$\dash subsets of the points. Then the algorithm below shows how to compute ${\mathcal L}_{k+1}$ from ${\mathcal L}_k$ with no explicit isomorphism tests between pairs of $(k+1)$\dash subsets.

\begin{framed}
\renewcommand{\labelitemi}{$\bullet$}
\renewcommand{\labelitemii}{$\diamond$}
\noindent{\sc Algorithm 1}\\\\
For each $k$\dash subset $X \in {\mathcal L}_k$ 
\begin{itemize}
\item Compute the group $G_X$ fixing $X$ setwise
\item For each orbit representative $x$ of $G_X$ on $\mathcal{P}_r \backslash X$
\begin{itemize}
\item Let $Y = X \cup \{x\}$
\item Compute the group $G_Y$ and the corresponding canonically labelled graph
\item Add $Y$ to $\mathcal{L}_{k+1}$ if and only if $x$ is in the same $G_Y$\dash orbit
as the lowest canonically labelled vertex of $Y$.
\end{itemize}
\end{itemize}
\end{framed}
\begin{center}Algorithm 1: An orderly progression from ${\mathcal L}_k$ to ${\mathcal L_{k+1}}$
\end{center}

\begin{theorem}
With the notation above, if ${\mathcal L}_k$ contains exactly one representative of each $G$\dash orbit on $k$\dash subsets of ${\mathcal P}_r$, then the set ${\mathcal L}_{k+1}$ produced by  Algorithm~$1$ contains exactly one representative of each $G$\dash orbit on $(k+1)$\dash subsets of ${\mathcal P}_r$.
\end{theorem}

\begin{proof}
Suppose that $Y$ is a $(k+1)$\dash subset of ${\mathcal P}_r$. We need to show that ${\mathcal L}_{k+1}$ contains exactly one isomorph of $Y$. Consider the canonically labelled version of $\Gamma_r$ with $Y$ distinguished, let $y$ be the element of $Y$ with the lowest canonical label and set $X = Y \backslash \{y\}$. Then by the
inductive hypothesis, ${\mathcal L}_k$ contains some isomorph of $X$. When this
isomorph is processed by Algorithm 1, all of its single-element extensions will be considered including an isomorph of $Y$ which will then be accepted. Hence an isomorph of $Y$ is accepted at least once. An isomorph of $Y$ can be accepted only as an extension of an isomorph of $X$, and hence different $k$\dash subsets cannot yield isomorphic $(k+1)$\dash subsets. Finally note that if both $X \cup x_1$ and $X \cup x_2$ are accepted as extensions of $X$, then they cannot be isomorphic. If they were isomorphic then some automorphism would map $x_1$ to $x_2$  (as they are both in the orbit of the lowest canonically labelled vertex), and hence some automorphism fixing $X$ would map $x_1$ to $x_2$, contradicting the fact that only one representative of each orbit of $G_X$ is considered for addition to $X$ as it is processed. \end{proof}

\subsubsection{Modification of this algorithm}
This algorithm is easily modified to produce only those matroids that exclude a given minor, say $M$. (Henceforth we will identify a matroid with the corresponding $k$\dash subset of ${\mathcal P}_r$ and mix the graph or matroid terminology interchangeably, even in a single sentence!)

If $M$ has $k$ elements, then the algorithm is run normally until the set ${\mathcal L}_k$ has been produced. One of these $k$\dash sets is equivalent to $M$ and we then set   ${\mathcal L}'_k = {\mathcal L}_k \ba M$ which, by definition, contains every $k$\dash element simple binary matroid of rank up to $r$ with no minor isomorphic to $M$. If Algorithm 1 is then applied to ${\mathcal L}'_k$ then the resulting list will certainly contain every $(k+1)$\dash element matroid without an $M$\dash minor, but probably will introduce some new matroids that {\em do} contain an $M$\dash minor. Therefore, a two-step process is used to produce ${\mathcal L}'_{k+1}$ from ${\mathcal L}'_k$:

\begin{enumerate}
\item Use Algorithm 1 to compute ${\mathcal L}'_k \rightarrow {\mathcal M}_{k+1}$
\item Form ${\mathcal L}'_{k+1}$  by removing any matroids with an $M$\dash minor from ${\mathcal M}_{k+1}$.
\end{enumerate}

The second stage of this process would be prohibitively expensive to perform if each matroid had to be directly tested for the presence of a minor isomorphic to $M$. However we can exploit the fact that the lists ${\mathcal L}'_\ell$ for $\ell \leq k$ jointly contain all smaller simple binary matroids with no minor isomorphic to $M$. As each candidate matroid is created, we test that (the simplifications) of all of its single-element deletions and contractions are contained in these lists, thereby certifying that the matroid is $M$\dash free.

Once the list of all simple prism-free binary matroids of rank up to 8 has been found, a straightforward calculation verifies that there are no \ifourc prism-free binary matroids of rank 6, 7 or 8 and hence Lemma~\ref{4} holds. A slightly more elaborate computation determining the minor order on the prism-free binary matroids confirms that Lemma~\ref{5} is also true.


\end{document}